\documentclass[11pt]{article} 

\usepackage{amsmath, amscd, amsthm, amssymb, mathrsfs, color, framed, pstricks, calc} 
\usepackage[applemac]{inputenc}
\usepackage[all]{xy} 
\usepackage[T1]{fontenc} 
\usepackage{textcomp} 
\usepackage[adobe-utopia]{mathdesign}

\DeclareMathOperator{\Diff}{Diff}

\DeclareMathOperator{\SO}{SO}

\DeclareMathOperator{\U}{U}
\DeclareMathOperator{\SU}{SU}

\DeclareMathOperator{\Lie}{Lie}

\DeclareMathOperator{\Stab}{Stab}
\DeclareMathOperator{\Ham}{Ham}
\DeclareMathOperator{\HVect}{HVect}

\DeclareMathOperator{\Map}{Map} 

\newcommand{\R}{\mathbb R}
\newcommand{\C}{\mathbb C}
\newcommand{\Z}{\mathbb Z}

\newcommand{\s}{\mathcal S}
\newcommand{\G}{\mathcal G}
\newcommand{\X}{\mathcal X}

\newcommand{\A}{\mathcal A}

\newcommand{\T}{\mathcal T}

\newcommand{\diff}{\text{\rm d}}
\newcommand{\del}{\partial}
\newcommand{\delb}{\bar{\del}}

\newcommand{\su}{\mathfrak{su}}

\newcommand{\g}{\mathfrak{g}}

\renewcommand{\P}{\mathbb P}

\renewcommand{\u}{\mathfrak{u}}

\theoremstyle{plain}
	\newtheorem{theorem}{Theorem}
	\newtheorem{proposition}[theorem]{Proposition}
	\newtheorem{lemma}[theorem]{Lemma}

\theoremstyle{definition}
	\newtheorem{definition}[theorem]{Definition}
	\newtheorem{remark}[theorem]{Remark}

\theoremstyle{plain}
	\newtheorem*{theorem*}{Theorem}
	\newtheorem*{proposition*}{Proposition}
	\newtheorem*{lemma*}{Lemma}
	\newtheorem*{corollary*}{Corollary}
	\newtheorem*{conjecture*}{Conjecture}
\theoremstyle{definition}
	\newtheorem*{definition*}{Definition}
	\newtheorem*{remark*}{Remark}
	\newtheorem*{remarks*}{Remarks}
	
\makeatletter
\def\blfootnote{\xdef\@thefnmark{}\@footnotetext}
\makeatother

\begin{document}

\title{The Hamiltonian geometry of the space of unitary connections with symplectic curvature}
\author{Joel Fine}
\date{ }

\maketitle

\begin{abstract}
Let $L \to M$ be a Hermitian line bundle over a compact manifold. Write $\s$ for the space of all unitary connections in $L$ whose curvatures define symplectic forms on $M$ and $\G$ for the group of unitary bundle isometries of $L$, which acts on $\s$ by pull-back. The main observation of this note is that $\s$ carries a $\G$-invariant symplectic structure, there is a moment map for the $\G$-action and that this embeds the components of $\s$ as $\G_0$-coadjoint orbits (where $\G_0$ is the component of the identity). Restricting to the subgroup of $\G$ which covers the identity on $M$, we see that prescribing the volume form of a symplectic structure can be seen as finding a zero of a moment map. When $M$ is a Kähler manifold, this gives a moment-map interpretation of the Calabi conjecture. We also describe some directions for future research based upon the picture outlined here.
\end{abstract}

\section{Introduction}

Let $L \to M$ be a Hermitian line bundle over a compact $2n$-dimensional manifold.  We assume throughout that $c_1(L)$ contains symplectic forms. This note investigates the space $\s$ of all unitary connections $A$ in $L$ for which $\omega_A = \frac{i}{2\pi}F_A$ is a symplectic form on $M$. The group $\G$ of unitary bundle isometries (not necessarily covering the identity on $M$) acts on $\s$ by pull-back. The main observation of this note is the following.

\begin{theorem}\label{main_observation}
~
\begin{itemize}
\item
$\s$ carries a $\G$-invariant symplectic form;
\item
There is an equivariant moment-map $\mu \colon \s \to \Lie (\G)^*$ for the $\G$-action;
\item
The map $\mu$ embeds each component of $\s$ as a coadjoint orbit of $\G_0$, the identity component of $\G$.
\end{itemize}
\end{theorem}

This is proved in \S\ref{proof_main_observation}. In \S\ref{integrality} we show that the coadjoint orbit of $A \in \s$ is integral if and only if the Weinstein homomorphism $\pi_1(\Ham) \to S^1$ is trivial (where $\Ham = \Ham(\omega_A)$ is the group of Hamiltonian diffeomorphisms).

In \S\ref{prescribe_volume} we consider the restriction of the moment map $\mu$ for the action of  the subgroup $\T = \Map(M, S^1) \subset \G$ of bundle isometries covering the identity on $M$. It turns out that the moment map sends a connection $A$ to the volume form $\omega_A^n/n!$. In this way the problem of prescribing the volume of a symplectic structure can be seen in terms of moment map geometry. 

As we explain in \S\ref{symplectic_case} one outcome of this is that when $b_1(M)=0$ the space of symplectic forms with fixed volume form is naturally a symplectic manifold. When $b_1(M) \neq 0$ this space carries a torus-fibration with fibres of dimension $b_1(M)$ whose total space is naturally a symplectic manifold.

In \S \ref{Kahler_case} we consider the problem of prescribing the volume form of a Kähler metric. This is the renowned Calabi conjecture, now of course Yau's theorem \cite{yau:otrcoackmatcmae1}. Using the picture outlined above we show how the Calabi conjecture can be phrased as finding a zero of the moment map inside a complex group orbit. This puts the problem into the same framework as the Hitchin--Kobayashi correspondence (concerning Hermitian--Einstein connections) and the Donaldson--Tian--Yau conjecture (concerning Kähler metrics with constant scalar curvature). 

The focus of this note is to explain the above geometric picture; no attempt is made here, however, to explore the potential applications. Both \S\ref{the_moment_map} and \S\ref{prescribe_volume} end with a brief discussion of some of these possible directions for future research (some more speculative than others!).

\subsection*{Acknowledgements}

It is a pleasure to acknowledge the influence of conversations with the following people:  Frédéric Bourgeois, Baptiste Chantraine, Dmitri Panov, Simone Gutt, Julien Keller and Chris Woodward.

This work was written up whilst I was a guest at the Simons Center for Geometry and Physics, at the State University of New York, Stony Brook. I am grateful for the hospitality and the stimulating research environment which they provided. 

\section{The space of connections with symplectic curvature}
\label{the_moment_map}

Recall that $L \to M$ is a Hermitian line bundle over a compact $2n$-dimensional manifold. We write $\s$ for the space of all unitary connections $A$ in $L$ for which $\omega_A = \frac{i}{2\pi}F_A$ is a symplectic form on $M$. 

\subsection{Symplectic structure and moment map}\label{proof_main_observation}

We begin by describing a symplectic structure on $\s$. The set $\s$ is open in the space of all connections (for, say, the $C^\infty$ topology). The tangent space $T_A \s$ is the space $\Omega^1(M, i\R)$ of imaginary 1-forms. In order to avoid factors of $\frac{i}{2\pi}$ in all our formulae, we multiply by $-2\pi i$ at the outset in identifying $T_A \s \cong \Omega^1(M, \R)$. Given $A \in \s$, we write $\omega_A = \frac{i}{2\pi}F_A$ for the associated symplectic form. Our conventions mean that for $a\in \Omega^1(M,\R)$ corresponds to an infinitesimal change of $\diff a$ in $\omega_A$.

\begin{definition}
We define a 2-form $\Omega$ on $\s$ by
\[
\Omega_A(a,b) = \frac{1}{(n-1)!}\int_X a \wedge b \wedge \omega_A^{n-1},
\]
for $a,b \in \Omega^1(M ,\R)$.
\end{definition}

\begin{proposition}
The 2-form $\Omega$ is a symplectic form.
\end{proposition}

\begin{proof}
To prove non-degeneracy on $T_A\s$, let $J$ be an almost complex structure on $M$ compatible with $\omega_A$. Then, for a non-zero 1-form $a$, 
\[
\Omega_A(a, Ja) = \frac{1}{n!}\int_X |a|^2\, \omega^n_A >0
\] 
where $|\cdot|^2$ is the Riemannian metric corresponding to $J$ and $\omega_A$.

Next we check $\Omega$ is closed. For this let $a, b, c \in \Omega^1(X,\R)$, thought of as vector fields on $\s$. Then
\[
\diff \Omega (a,b,c)
=
a \cdot \Omega(b,c) + b \cdot \Omega(c,a) + c \cdot \Omega(a,b).
\]
(The formula for general vector fields also includes terms with Lie brackets, but in our case these vanish since the vector fields $a,b,c$ are linear on the affine space of all connections and so commute.) Now, 
\[
a \cdot  \Omega(b,c)
=
\frac{1}{(n-2)!}
\int_M 
b \wedge c \wedge  \diff a \wedge  \omega_A^{n-2}.
\]
Hence
\begin{eqnarray*}
\diff\Omega(a,b,c)
&=&
\frac{1}{(n-2)!}
\int_M
\left(
\diff a \wedge b \wedge c  + \diff b \wedge c \wedge a
+ \diff c  \wedge a \wedge b\right)
\wedge
\omega_A^{n-2},\\
&=&
\frac{1}{(n-2)!}
\int_M
\diff\left(a\wedge b \wedge c \wedge \omega_A^{n-2}\right),\\
&=&
0.
\end{eqnarray*}
\end{proof}

We write $\G$ for the group of bundle isometries of $L$, not necessarily covering the identity on $M$. $\G$ acts by pull-back on $\s$, preserving $\Omega$. To describe the moment map for this action, we first note that given a connection $A$ in $L$ and $\eta \in \Lie (\G)$, one can define a function $A(\eta) \in C^\infty(M, \R)$. Thinking of $\eta$ as a vector field on $L$, the connection $A$ splits $\eta$ into a vertical and a horizontal part. On each fibre, the vertical part is multiplication by $\frac{i}{2\pi}A(\eta)$.

Alternatively, we can think of a connection $A$ as an $S^1$-invariant 1-form on the principal circle bundle $P \to M$ corresponding to $L \to M$. Then $\eta$ is an $S^1$-invariant vector field on $P$ and the function $A(\eta)$ given by pairing the 1-form $A$ with the vector field $\eta$ is the function we seek, pulled back to ~$P$. (Again, normally one considers connections on principal circle bundes as \emph{imaginary} valued 1-forms, but we multiply by $-2\pi i$ throughout and use instead real 1-forms.) This second point of view---via principal bundles---is the one we normally adopt in this section.

\begin{proposition}\label{moment_map_G}
The map $\mu \colon \s \to \Lie(\G)^*$ defined by
\[
\langle \mu(A), \eta \rangle
=
\frac{1}{n!}\int_M A(\eta)\, \omega_A^n
\]
is a $\G$-equivariant moment map for the action of $\G$ on $\s$.
\end{proposition}

\begin{proof}
Given $\eta \in \Lie (\G)$, let $a_\eta \in \Omega^1(M, \R)$ be the vector field on $\s$ corresponding to the infinitesimal action of $\eta$. Let $b \in \Omega^1(M,\R)$ be another vector field on $\s$. The identity to be proved is
$
b \cdot \langle \mu, \eta \rangle
=
\Omega(b,a_\eta)
$.

We begin with the left-hand-side. We use the description in terms of the principal $S^1$-bundle $p \colon P \to M$ given above, in which $A$ is regarded as an $S^1$-invariant 1-form on $P$. The vector field $b \in \Omega^1(M, \R)$ on $\s$ corresponds to an infinitesimal change of $p^*b$ in $A$ and hence an infinitesimal change of $p^*b(\eta) = b(p_*\eta)$ in $A(\eta)$. Meanwhile, the infinitesimal change in $\omega_A$ is $\diff b$. Hence,
$$
b \cdot \langle \mu, \eta \rangle
=
\int_M
\left(
\frac{1}{n!}b(p_*\eta)\, \omega_A^n
+ 
\frac{1}{(n-1)!}A(\eta)\,\diff b \wedge \omega_A^{n-1}
\right).
$$

To compute the right-hand-side of the moment-map identity, still thinking of $A$ as a 1-form on $P$, we have that
\begin{eqnarray*}
a_\eta
	&=&
		\mathcal{L}_\eta(A)\\
	&=& 
		(\diff \circ \iota_\eta + \iota_\eta \circ \diff)A,\\
	&=&
		\diff(A(\eta)) + \iota_{p_*\eta} \omega_A.
\end{eqnarray*}
(We have implicitly identified $a_\eta \in \Omega^1(M, \R)$ and $p^*a_\eta \in \Omega^1(P, \R)$ in the first two lines here.) Hence, evaluated at the point $A \in \s$,
$$
\Omega(b,a_\eta)
=
\frac{1}{(n-1)!}\int_M
b \wedge 
\left(\diff (A(\eta)) + \iota_{p_*\eta} \omega_A\right)
\wedge 
\omega_A^{n-1}.
$$
Next we use the following identity:  on a $2n$-dimensional manifold, given a 1-form $\alpha$ and a 2-form $\beta$ the $(2n+1)$-form $\alpha \wedge \beta^n$ necessarily vanishes; hence, for any vector field $v$,
\[
0 = \iota_v (\alpha \wedge \beta^n) = \alpha(v) \beta^n- n \alpha \wedge \iota_v\beta \wedge \beta^{n-1}.
\]
Putting $\alpha = b$, $\beta = \omega_A$ and $v = p_*\eta$, this gives
\begin{eqnarray*}
\Omega(b,a_\eta)
&=&
\frac{1}{(n-1)!}\int_M \left(
b \wedge \diff (A(\eta)) \wedge \omega_A^{n-1}
+
\frac{1}{n} b(p_*\eta) \omega_A^n
\right),\\
&=&
\int_M\left(
\frac{1}{(n-1)!}A(\eta) \diff b \wedge \omega_A^{n-1} 
+
\frac{1}{n!} b(p_*\eta)\, \omega_A^n
\right),\\
&=&
b \cdot \langle \mu ,\eta \rangle.
\end{eqnarray*}

Finally, $\G$-equivariance follows immediately from the  definition of~$\mu$.
\end{proof}

We remark that this picture is motivated by the well-known observation of Atiyah and Bott \cite{atiyah-bott} that ``curvature is a moment map''.  In \cite{atiyah-bott}, Atiyah and Bott consider unitary connections in bundles of \emph{arbitrary} rank, but over a base with a \emph{fixed} symplecitc form.

To complete the proof of Theorem \ref{main_observation} we show that the components of $\s$ are identified via $\mu$ with coadjoint orbits.

\begin{lemma}\label{transitive}
The map $\mu \colon \s \to \Lie(\G)^*$ embeds each component of $\s$ as a coadjoint orbit of $\G_0$.
\end{lemma}
\begin{proof}
We must show two things: firstly, that $\mu$ is injective; secondly that $\G_0$ acts transitively on the components of $\s$.

To prove injectivity of $\mu$, suppose that $A \neq A'$. Then we can find a vector field $v$ on $M$ such that the $A'$-horizontal lift $\eta$ of $v$ satisfies $A(\eta)>0$, hence $\langle \mu(A), \eta \rangle >0$. But $A'(\eta) = 0$ and so $\langle \mu(A'), \eta \rangle =0$, hence $\mu(A) \neq \mu(A')$.

Next we show that $\G_0$ acts transitively on the components of $\s$. Given $A\in \s$, let $\rho_A \colon \Lie(\G) \to T_A\s$ denote the infinitesimal action of $\G$ at $A$. We have already seen that
\[
\rho_A(\eta)
=
a_\eta
=
\diff(A(\eta))+ \iota_{p_*\eta}\omega_A.
\]
First we show that $\rho_A$ is surjective. Given $a \in \Omega^1(M, \R)$, let $v$ be the $\omega_A$-dual vector field and let $\eta$ be the $A$-horizontal lift of $v$ to $P$. Then $\rho_A(\eta) = a$.

Now, given a path $A(t)$ in $\s$, let $v(t)$ be the vector field which is $\omega_{A(t)}$-dual to $\frac{\diff A}{\diff t}(t)$ and let $\eta(t)$ be the $A(t)$-horizontal lift of $v(t)$ to $P$. The time-dependent vector field $\eta(t)$ integrates up to a path $g(t)$ in $\G_0$ with $g(0)$ the identity. By construction, $g(t) \cdot A(0) = A(t)$. 
\end{proof}

\subsection{Integrality and the Weinstein homomorphism}\label{integrality}

We next turn to the question of whether or not the orbits of $\s$ are \emph{integral} coadjoint orbits. It turns out that the obstruction to this is a homomorphism $\pi_1(\Ham_A) \to S^1$, first introduced by Weinstein \cite{weinstein}.

We briefly recall the definition of an integral coadjoint orbit. For more details see, for example, \cite{kirillov}. Given a Lie group $G$ with Lie algebra $\g$, fix $f \in \g^*$. We write $\Stab(f) \subset G$ for the stabiliser of $f$ under the coadjoint action and $\mathfrak{h}$ for the Lie algebra of the stabiliser. The linear map $f \colon \g \to \R$ restricts to a Lie algebra homomorphism $f \colon \mathfrak h \to \R$. The orbit $\mathcal O_f$ of $f$ is called \emph{integral} when the map $\mathfrak h \to \R$ is (up to a factor of $i$) the derivative of a group homomorphism $\Stab(f) \to S^1$. This condition implies the existence of a line bundle $L \to \mathcal O_f$ which carries a connection whose curvature is the symplectic form on $\mathcal O_f$; moreover the symplectic action of $G$ on $\mathcal O_f$ lifts to a connection-preserving action on $L$. 

Accordingly, we next investigate the stabiliser $\Stab_A \subset \G_0$ of a point $A \in \s$. For an alternative exposition of the following, see Weinstein's article \cite{weinstein}.

We start from the a short exact sequence
\[
1 \to \Map_0(M, S^1) \to \G_0 \to \Diff_0(M) \to 1
\]
(where the subscripts $0$ denote the identity components.) 

\begin{lemma}[Weinstein \cite{weinstein}]
Restricting this sequence to $\Stab_A \subset \G_0$ gives a short exact sequence
\begin{equation}\label{ses_for_stab}
1 \to S^1 \to \Stab_A \to \Ham_A \to 1
\end{equation}
where $S^1 \subset \Map_0(M, S^1)$ are the constant gauge transformations. 
\end{lemma}
\begin{proof}
First note that the restriction of the map $\G_0 \to \Diff(M)$ to $\Stab_A$ certainly takes values in $\omega_A$-symplectomorphisms. To verify that the image lies in $\Ham_A$, recall the formula for the infinitesimal action $\rho_A(\eta)$ of $\eta \in \Lie(\G)$ at $A$ given above. From this it follows that $\eta \in \Lie (\Stab_A)$ if and only if $p_*\eta$ is a Hamiltonian vector field with Hamiltonian $-A(\eta)$. 

Next we check that the map $\pi \colon \Stab_A \to \Ham_A$ is surjective. Given a $\omega_A$-Hamiltonian vector field $v$ on $M$ with Hamiltonian $h$ we write $v^\flat$ for the $A$-horizontal lift of $v$. Then the vector field $\eta = v^\flat - h \frac{\del}{\del\!\theta}$ on $P$ is $S^1$-invariant, hence in $\Lie(\G)$ and $\rho_A(\eta) = 0$. So $\eta \in \Lie(\Stab_A)$ and $\pi_*(\eta) = v$, meaning $\pi_*$ is surjective. Integrating this shows that $\pi \colon \Stab_A \to \Ham_A$ is  surjective.

The kernel of $\pi$ is $\Stab_A \cap \Map_0(M, S^1)$. Given $f \colon M \to S^1$, the corresponding change in $A$ is $f \diff( f^{-1})$.  Hence $\ker \pi =S^1$ is the constants, and the short exact sequence for $\G_0$ restricts to $\Stab_A$ as claimed.
\end{proof}

Given $A \in \s$ the moment map at $A$ restricts to give a Lie algebra homomorphism
\[
\mu(A) \colon \Lie(\Stab_A) \to \R
\]
The kernel of this map is an ideal $I \subset \Lie(\Stab_A)$; moreover, the inclusion $S^1 \subset \Stab_A$ determines a copy of $\R \subset \Lie(\Stab_A)$ which is mapped isomorphically onto $\R$ by $\mu$. It follows that the derivative of $\Stab_A \to \Ham(A)$ identifies $I \cong \HVect_A$ and so there is a splitting
\begin{equation}
\label{lie_alg_splitting}
\Lie(\Stab_A) \cong \R \oplus \HVect_A
\end{equation}
into a direct sum of ideals.

Using left-multiplication we can view the splitting (\ref{lie_alg_splitting}) as defining a connection on the principle $S^1$-bundle $\Stab_A \to \Ham_A$. Because the horizontal subspace (the $\HVect_A$ summand) is a Lie sub-algebra of $\Lie(\Stab_A)$, this connection is flat. Its holonomy is the \emph{Weinstein homomorphism},
\[
w \colon \pi_1(\Ham_A) \to S^1.
\]

\begin{proposition}
Given $A \in \s$, the corresponding coadjoint orbit of $\G_0$ is integral if and only if the Weinstein homomorphism $w \colon \pi_1(\Ham_A) \to S^1$ is trivial.
\end{proposition}
\begin{proof}
The coadjoint orbit of $A$ is integral precisely when the kernel of the homomorphism $\mu(A) \colon \Lie(\Stab_A) \to \R$ integrates up to a subgroup of $\Stab_A$. In our case this kernel defines the horizontal space of the flat connection whose holonomy is $w$. So the orbit is integral if and only if parallel transport identifies all the $S^1$-fibres of $\Stab_A \to \Ham_A$. This happens precisely when the holonomy is trival.
\end{proof}

On the one hand, there are examples of symplectic manifolds for which the Weinstein homomorphism is trivial. Indeed, for a surface of genus at least one, the Hamiltonian group is even contractible. On the other hand, there are also plenty of manfiolds for which the Weinstein homomorphism is non-trivial; the simplest being $S^2$. To see this, restrict the short exact sequence (\ref{ses_for_stab}) to the subgroup $\SO(3) \subset \Ham$ to obtain the sequence
\[
1\to S^1 \to \U(2) \to \SO(3) \to 1.
\]
The flat connection corresponds to the Lie algebra isomorphism $\u(2) \cong \su(2)\oplus i\R$; its holonomy is non-trivial and gives the standard isomorphism
\[
\U(2) \cong \SU(2) \times_{\pm1} S^1.
\]
Similiar remarks apply to $\C\P^n$ with the Fubini--Study metric and, more generally to certain toric varieties. See the recent survey article of McDuff \cite{mcduff} for more on this subject.

\subsection{Further questions}

Given a subgroup $\mathcal H \subset \Diff_0(M)$, the preimage under $\G_0 \to \Diff_0(M)$ is a subgroup $\mathcal H' \subseteq \G_0$ which inherits a Hamiltonian action on $\s$. The moment-map $\mu'$ for the action of $\mathcal H'$ is simply the projection of $\mu$ under $\Lie (\G_0)^* \to \Lie (\mathcal H')^*$. One might look for zeros of $\mu'$ in the hope that they give symplectic structures which respect in some way the additional geometry imposed in passing from $\Diff_0(M)$ to $\mathcal H$. 

We explore this idea in the next section in its most extreme form, when $\mathcal H = 1$ is the trivial group. This leads to the problem of prescribing the volume form of a symplectic structure. In a forthcoming paper \cite{fine-asdE} we exploit this same idea for certain manifolds $M$ and subgroups $\mathcal H$. The manifolds in question are $S^2$-bundles over four-manifolds and in this way we give a moment-map interpretation of the anti-self-dual Einstein equations for a Riemannian metric on a four-manifold. Besides these two situations, however, there are many other possibilities one could study and it would be interesting to see more examples. 

We close this section with a speculative remark. The above picture associates to each isotopy class of symplectic forms in $c_1(L)$ a certain coadjoint orbit of $\G_0$. On the one hand, distinguishing isotopy classes of symplectic forms is a central problem in symplectic topology; on the other hand, distinguishing coadjoint orbits is a central problem in the theory of infinte dimensional Lie groups. One might hope that Theorem \ref{main_observation} opens up the path for a transfer of ideas between these two as yet poorly understood questions. 

An important approach to the study of coadjoint orbits is the celebrated ``orbit method'' (see for example the text of Kirillov \cite{kirillov}). For the group $\G_0$, perhaps the first case to consider would be a surface of genus at least one. There, the corresponding coadjoint orbit is integral. Moreover, as we will see in the following section, it comes with a natural isotropic fibration whose infinite-dimen\-sional fibres fail to be coisotropic by a finite dimensional discrepancy (see Remark \ref{lag_fib_remark}). Thus we have in place more-or-less the initial data required by geometric quantisation. This still leaves, of course, the principal difficulty of what should play the rôle of the ``square-integrable sections'' of the prequantum line bundle, since the base is infinite dimensional. Exactly how to quantise such a coadjoint orbit is, in my opinion at least, an interesting and difficult question.

\section{Prescribing the volume form of a symplectic structure}\label{prescribe_volume}

Given a Hamiltonian action of a group $G$ with a moment map $\mu$ taking values in $\g^*$, the action of a sub-group $H \subset G$ has moment map given by composing $\mu$ with the projection $\g^* \to \mathfrak{h}^*$. In this section we apply this observation to the action of the subgroup $\T \subset \G$ of bundle isometries of $L \to M$ which cover the identity. 

\subsection{Purely symplectic case}\label{symplectic_case}

Of course, $\T = \Map(M, S^1)$ and so $\Lie(\T) = C^\infty(X, \R)$. (One normally uses imaginary valued functions here but again we have multiplied by $-2\pi i$ throughout.) By integrating against top-degree forms, we can identify $\Omega^{2n}(M, \R)$ with a subset of $\Lie(\T)^*$. With this understood, we have the following result, which is an immediate corollary of Proposition \ref{moment_map_G}.

\begin{proposition}
There is an equivariant moment map $\nu \colon \s \to \Lie(\T)^*$ for the action of $\T$ on $\s$ given by $\nu(A) =\omega_A^n/n!$
\end{proposition}

So prescribing the volume of a symplectic structure in $c_1(L)$ can be seen as finding a zero of a moment map. More precisely, since $\T$ is abelian, the coadjoint action is trivial and so we can equally use $\nu - \theta$ as a moment map for any $\theta \in \Lie(\T)^*$. Given a volume form $\theta \in \Omega^{2n}(M, \R)$ with $[\theta] = \frac{1}{n!} c_1(L)^n$, the equation for $A \in \s$ given by $\omega_A^n/n! = \theta$ is the same as finding a zero of the moment map~$\nu - \theta$.

Given such a $\theta$, we next turn to the symplectic reduction $\nu^{-1}(\theta)/\T$. By standard theory this is a symplectic manifold (of infinite dimension). To describe it we write $\X_\theta$ for the space of symplectic forms $\omega \in c_1(L)$ with $\omega^n/n! = \theta$.

\begin{proposition}\label{fixed_volume_form}
If $b_1(M)= 0$ then $\X_\theta = \nu^{-1}(\theta)/\T$ and so, in particular, the space of symplectic forms with fixed volume form is naturally a symplectic manifold. In general there is a submersion from the symplectic reduction $\nu^{-1}(\theta)/\T \to \X_\theta$ with fibres isomorphic to $H^1(M,\R)/H^1(M,\Z)$. The restriction of the symplectic structure to these fibres is identified with the 2-form on $H^1(M,\R)$ defined by $(\alpha, \beta) \mapsto \frac{1}{(n-1)!}\int_M \alpha \wedge \beta \wedge c_1(L)^{n-1}$.
\end{proposition}

\begin{proof}
We begin with the following standard fact. Given a symplectic form $\omega \in c_1(L)$, write $\s_\omega\subset \s$ for the set of unitary connections $A$ for which $\omega_A = \omega$. Then $\s_\omega/\T$ can be identified with $H^1(M, \R)/H^1(M,\Z)$. 

More precisely, given $A_0 \in \s_\omega$, any other connection $A \in \s_\omega$ is of the form $A = A_0 + \frac{i}{2\pi}a$ for a closed 1-form $a$. There is thus a surjection $c \colon \s_\omega \to H^1(M,\R)$ given by $c(A)= [a]$. Now $\T = \Map(M, S^1)$ acts on $H^1(M, \R)$, the action of $f \in \T$ on $H^1(M, \R)$ is by addition of $\frac{1}{2\pi i}[f \diff(f^{-1})] \in H^1(M, \Z)$. With this action understood, $c$ is $\T$-equivariant. Since any element of $H^1(M, \Z)$ can be written in as $\frac{1}{2\pi i}[f \diff(f^{-1})]$ for some $f \in \T$, the map $c$ descends to an identification $\s_\omega/\T \to H^1(M,\R)/H^1(M,\Z)$. 

The group $\T$ is abelian, so its orbits in $\s$ are isotropic and hence the restriction of the symplectic form $\Omega$ on $\s$ to $\s_\omega$ descends to a 2-form on $\s_\omega/\T \cong H^1(M, \R)/H^1(M, \Z)$. It follows from the definition of $\Omega$ that the 2-form is identified with the 2-form on $H^1(M, \R)$ given by
\[
(\alpha, \beta) \mapsto \frac{1}{(n-1)!}\int_M \alpha \wedge \beta \wedge c_1(L)^{n-1}.
\]

The result follows from these two observations applied fibrewise to the map $\nu^{-1}(\theta) \to \X_\theta$ which sends each connection $A$ to its curvature $\omega_A$.
\end{proof}

\begin{remark}
When $b_1(M) = 0$, the symplectic structure on $\X_\theta$ can be seen directly (and with no need for the condition that the fixed choice of symplectic class be integral). The tangent space at a point $\omega \in \X_\theta$ is the space of exact 2-forms $\gamma$ such that $\omega^{n-1} \wedge \gamma = 0$. We now define a skew pairing $\Theta$ on $T_\omega \X_\theta$ by
\[
\Theta(\gamma, \gamma') 
=
\frac{1}{(n-1)!}
\int_M a \wedge a' \wedge \omega^{n-1}
\]
where $a, a'$ are 1-forms with $\diff a= \gamma$, $\diff a' = \gamma'$. If $\tilde a$ is another 1-form with $\diff \tilde a = \gamma$, then $\diff(a - \tilde a) = 0$ and so, since $b_1(M) = 0$, we can write $a - \tilde a = \diff f$ for some function $f$. Hence,
\[
\int_M (a - \tilde a) \wedge a' \wedge \omega^{n-1}
=
-\int_M f \diff a' \wedge \omega^{n-1}
\]
which vanishes since $\diff a' \wedge \omega^{n-1} = \gamma' \wedge \omega^{n-1} = 0$. It follows that $\Theta(\gamma, \gamma')$ does not depend on the choice of $a$ or $a'$.

When the fixed symplectic class $[\omega] = c_1(L)$ is integral, $\Theta$ is precisely the 2-form which arises from the identification $\nu^{-1}(\theta)/\T \cong \X_\theta$. It follows from the general theory that $\Theta$ is closed and non-degenerate, something which one can verify directly from the definition.
\end{remark}

\begin{remark}
Still under the assumption that $b_1(M) = 0$, note that the group $\Diff(M, \theta)$ of volume-pre\-serving diffeomorphisms acts on the symplectic manifold $\X_\theta$. This action is Hamiltonian in the sense that the infinitesimal action of a single divergence-free vector field $u$ is a Hamiltonian vector field on $\X_\theta$. To define a Hamiltonian $h \colon \X_\theta \to \R$ for the action of $u$ note that $L_u \theta = 0$ so $\iota_u \theta$ is a closed $(2n-1)$-form. Since $b_1(M)= b_{2n-1}(M)=0$ we can write $\iota_u \theta = \diff \beta$ for some $(2n-2)$-form $\beta$. We define the function $h$ by
\[
h(\omega) = - \int_M \beta \wedge \omega.
\]
Given a tangent vector $\gamma \in T_\omega \X_\theta$, i.e., an exact 2-form $\gamma = \diff a$ with $\gamma \wedge \omega^{n-1} = 0$, then the corresponding infinitesimal change in $h$ is given by
\[
\gamma \cdot h 
= 
- \int_M \beta \wedge \gamma
=
- \int_M \beta \wedge \diff a
=
\int_M \iota_u \theta \wedge a.
\]
On the other hand, the infinitesimal action of $u$ at $\omega \in \X_\theta$ is $\gamma_u= \diff (\iota_u \omega)$ and so
\[
\Theta(\gamma_u , \gamma) 
= 
\frac{1}{(n-1)!}\int_M \iota_u \omega \wedge a \wedge \omega^{n-1}
=
\int \iota_u \theta \wedge a
\]
Hence $h$ is a Hamiltonian for the action of $u$. 

Of course the Hamiltonian $h$ is uniquely determined only up to the addition of a constant. This is reflected in our description of $h$ by the freedom in the choice of $\beta$;  adding a closed $(2n-2)$-form to $\beta$ does not alter $\diff \beta = \iota_u\theta$ but changes $h$ by a constant. Writing down a moment map for the action amounts to choosing these constants consistently. The choices involved suggest that this cannot be done in such a way as to give an \emph{equivariant} moment map.
\end{remark}

\begin{remark}\label{lag_fib_remark}
As mentioned above, $\T$ is abelian and so the $\T$-orbits in $\s$ are isotropic. It follows from the standard theory of symplectic reduction that the fibres of the moment map $\nu$ are coisotropic and, moreover, given $A \in \s$, the tangent space to the fibre of $\nu$ through $A$ is the symplectic complement of the tangent space to the $\T$-orbit through $A$. When $M$ is a \emph{surface}, the isotropic fibration of $\s$ given by the $\T$-orbits is close to being a Lagrangian fibration. To see this, note that for a surface a volume forms and symplectic forms are the same thing. Now in the proof of Proposition \ref{fixed_volume_form} we saw that the codimension of $\T \cdot A$ in $\nu^{-1}(\omega_A)$ is $b_1(M)$. So for $S^2$ the $\T$-orbits give a Lagrangian fibration of $\s$, whilst for higher genus surfaces this infinite dimensional isotropic fibration fails to be Lagrangian only by a finite dimensional discrepancy. 
\end{remark}

\subsection{The Kähler case}\label{Kahler_case}

This point of view has additional use when $M$ is a complex manifold. Recall that the Calabi conjecture (now, of course, Yau's theorem \cite{yau:otrcoackmatcmae1}) states that given a Kähler class $\kappa \in H^2(M, \R)$ and volume-form $\theta$ on $M$ with total volume $\frac{1}{n!} \int_M\kappa^n$ there is a unique Kähler metric $\omega \in \kappa$ with $\omega^n/n! = \theta$. At least when $\kappa = c_1(L)$ is the first Chern class of a holomorphic line bundle, we can reformulate this problem as the search for the zero of a moment map in a complex group orbit, in a manner analogous to the Hitchin--Kobayashi correspondence \cite{donaldson:asdymcocasasvb,uhlenbeck.yau:oteohymcisvb} or the Donaldson--Tian--Yau conjecture concerning existence of constant scalar curvature Kähler metrics (as outlined in, for example, \cite{donaldson:scasotv}).

To describe this we first restrict attention to the subspace $\s^{1,1} \subset \s$ of unitary connections in $L \to M$ whose curvature is a positive $(1,1)$-form on the complex manifold $M$. The complex structure $J$ on $M$ makes $\s^{1,1}$ into a Kähler manifold. To see this notice that the endomorphism $a \mapsto Ja$ of $\Omega^1(M, \R)$ makes $\s$ into an almost complex manifold. Given $A \in \s^{1,1}$, $a \in T_A \s^{1,1}$ if and only if $\delb(a^{0,1}) = 0$. Since $(Ja)^{0,1} = - i a^{0,1}$ it follows that $\s^{1,1}$ is an almost complex submanifold of $\s$. 

To show that this almost complex structure is integrable we use the standard identification of the space of unitary connections having curvature of type $(1,1)$ with the space  of holomorphic structures on the line bundle $L \to M$. The identification sends a unitary connection $A$ to the $\delb$-operator given by the $(0,1)$-component of $A$. The $\delb$-operator is integrable precisely because $\omega_A$ is $(1,1)$. A unitary connection is determined by its $(0,1)$-component and, conversely, every integrable $\delb$-operator can be completed in a unique way to a unitary connection with $(1,1)$ curvature (see, e.g., \cite{griffiths.harris:poag}). In this way we identify $\s^{1,1}$ with the open subset of integrable $\delb$-operators whose curvatures are in fact Kähler forms. Under this identification, the almost complex structure described in the preceding paragraph is identified with the natural holomoprhic structure on the space of  integrable $\delb$-operators  

The symplectic structure $\Omega$ on $\s$ restricts to a Kähler metric on $\s^{1,1}$: given $A \in \s^{1,1}$, $\omega_A$ and $J$ pair to give a Riemannian metric $g_A$ on $M$; now $\Omega_A(a, Ja) =\|a\|^2$ is the $L^2(g_A)$-norm of $a$ and so the restriction of $\Omega$ to $\s^{1,1}$ pairs with the complex structure to give a Kähler metric on $\s^{1,1}$. 

This whole set-up is, of course, reminiscent of the moment-map description of the Hitchin--Kobayashi correspondence.  There one starts with a Hermitian vector bundle (of arbitrary rank) $E \to M$ and considers the space $\A^{1,1}$ of \emph{all} unitary connections with $(1,1)$-curvature or, equivalently, \emph{all} integrable $\delb$-operators. The key difference is that for the Hitchin--Kobayashi correspondence the symplectic structure on $\A^{1,1}$ is defined via a \emph{fixed} choice of Kähler metric on $M$. In our situation, however, the Kähler form $\omega_A$ on $M$ depends on the unitary connection $A \in \s^{1,1}$ and the symplectic structure on $\s^{1,1}$ is different from that in the Hitchin--Kobayashi correspondence.

Whilst the whole group $\G$ does not act by Kähler isometries on $\s^{1,1}$ (since the induced action on $M$ does not preserve $J$) the subgroup $\T = \Map(M, S^1)$ does. The action extends, at least locally, to a holomorphic (though not isometric) action of the complexification $\T^\C = \Map(M, \C^*)$. This is most easily seen by considering $\s^{1,1} \subset \mathcal H$ as an open set in the space of integrable $\delb$-operators in $L$. Now $\T^\C$ acts on $\mathcal H$ by pulling back. Note this is \emph{not} the same as pulling back the corresponding unitary connection by an element of $\T^\C$, since this does not preserve the property of being unitary. In terms of connections, the action of $f \in \Map(M, \C^*)$ on $A$ is given by 
\begin{equation}\label{action_of_TC}
f \cdot A = A + f \delb(f^{-1}) - \bar{f} \del (\bar f^{-1}).
\end{equation}
In particular, given a function $\phi \in C^\infty(M, \R)$, the action of $f = e^\phi$ is 
\begin{equation}\label{action_exp(phi)}
e^\phi \cdot A = A + \del \phi - \delb \phi
\end{equation}
and hence 
\begin{equation}\label{kahler_potential}
\omega_{e^\phi \cdot A}
=
\omega_A + \frac{i}{2\pi} \delb\!\del \phi.
\end{equation}
From this formula it is clear that the $\T^\C$-orbit of $A \in \s^{1,1}$ leaves the open set $\s^{1,1} \subset \mathcal H$. Indeed $e^\phi \cdot A$ remains in $\s^{1,1}$ precisely when $\phi$ is a Kähler potential for $\omega_A$. Nonetheless this calculation proves the following result.

\begin{lemma}
Fix $A_0 \in \s^{1,1}$. The map $A \mapsto \omega_A$ gives a surjection from $(\T^\C \cdot A_0) \cap \s^{1,1}$ to the space of Kähler metrics in $c_1(L)$.
\end{lemma}

From here we see that the Calabi conjecture fits into the general framework of moment maps in Kähler geometry. Namely, finding a Kähler form in $c_1(L)$ with volume form $\theta$ is the same as finding a zero of the moment map $\nu - \theta$ in a given complex orbit $(\T^\C \cdot A)\cap \s^{1,1}$. 

We recall a little of the general set-up alluded to here. The starting point is the action of a Lie group $G$ by holomorphic isometries on a Kähler manifold $X$, along with an equivariant moment map $\mu \colon X \to \g^*$. We suppose that the action extends to an action of $G^\C$, the complexifictaion of $G$. The problem is, given $x \in X$, to find $g \in G^\C$ such that $\mu(g \cdot x) = 0$. Since $\mu$ is $G$-invariant, this is really a question on the symmetric space $G^\C/G$. There is a function $F \colon G^\C/G \to \R$, called the Kempf--Ness function, whose critical points correspond to solutions of $\mu(g \cdot x) = 0$. Moreover, $F$ has the important property that it is convex along geodesics in $G^\C/G$. The downward gradient flow of $F$ provides a concrete way to attempt to find a zero of the moment map.

Applying this to the case of the $\T$-action on $\s^{1,1}$, we can give a moment map interpretation of some well-known facts concerning the Calabi conjecture. For a start, the symmetric space of interest is the quotient $\T^\C/\T$ of positive real functions $C^\infty(M, \R_+)$, or at least the open subset corresponding to $(\T^\C \cdot A) \cap \s^{1,1}$. Taking logarithms as in the discussion surrounding equations (\ref{action_exp(phi)}) and (\ref{kahler_potential}), we identify this space with the space of Kähler potentials
\[
\mathcal K = \left\{ \phi \in C^\infty(M, \R) : \omega_A + \frac{i}{2\pi} \delb\!\del \phi > 0\right\}.
\] 

Since $\T$ is an \emph{abelian} group, the symmetric metric on $\mathcal K$ should be flat. Indeed, tangent vectors correspond to infinitesimal Kähler potentials and, given our fixed choice of volume form $\theta$, the metric is given by the $L^2$ inner-product
\[
\langle f, g \rangle = \int_X fg\,\theta.
\]
In particular, the geodesics for this metric are simply the affine lines in $\mathcal K \subset C^\infty(M, \R)$. 

The Kemp--Ness function is determined by the requirement that when it is pulled back to a function on $G^\C$ its derivative in the imaginary directions is given by the moment map. So, in our situation, given $A \in \s^{1,1}$, the derivative of the pull-back of $F$ along the path $e^{t\phi} \cdot A$ is
\begin{equation}\label{energy_function}
\diff F (\phi) = \int_M \phi \left(\frac{\omega_A^n}{n!} - \theta\right). 
\end{equation}
As we saw above, on the level of Kähler forms, the tangent to the path $e^{t\phi}$ corresponds to the Kähler potential $\frac{i}{2\pi} \delb\!\del \phi$. So we can interpret $F$ as a function on the space of Kähler potentials, given by integrating (\ref{energy_function}) along a path. But this is precisely the definition of a well-known energy functional, the so-called ``$F_0$-functional'', described in, for example, \cite{tian:cmikg}. The standard moment-map theory tells us that $F$ is convex along affine lines in $\mathcal K$, something which can be verified directly. As has long been observed, this fact plays an important rôle in the study of the Calabi conjecture. In particular, since any two  points of $\mathcal K$ lie on a geodesic, we see immediately that a solution to the Calabi conjecture must be unique. 

We can also consider the downward gradient flow of $F$. In our situation, this is the flow of Kähler metrics given by
\begin{equation}\label{flow}
\frac{\del\!\omega}{\del\!t}
=
-\frac{i}{2\pi} \delb\!\del\left(\frac{\omega^n/n!}{\theta}\right)
\end{equation}
Given Yau's solution to the Calabi conjecture, one might expect that the flow (\ref{flow}) exists for all time and converges at infinity to the solution. This has very recently been proved by Cao--Keller \cite{cao-keller} and independently by Fang--Lai--Ma \cite{fang-lai-ma}. 

\subsection{Further questions}

Despite the fact that Yau has long since resolved the Calabi conjecture, this mo\-ment-map picture does raise interestings question. Typically there is a notion of ``stability'' associated to such a set-up; one then aims to show that a complex orbit is stable if and only if it contains a zero of the moment map. In our case, given a volume form $\theta$ we might hope to define the ``$\theta$-stability'' of $L \to X$. The general set-up would lead us to believe that $L \to X$ is $\theta$-stable (whatever that may mean) if and only if $c_1(L)$ contains a solution to the Calabi conjecture. Of course, we know that this is always the case and so perhaps the sought-after definition of $\theta$-stability is something trivially satisfied by all positive bundles $L \to X$. On the other hand, Yau's solution of the Calabi conjecture is a deep result, so one might optimistically speculate that $\theta$-stability (if indeed it can be defined) is some non-trivial property of $L \to X$ implied by Yau's theorem.

There are also other versions of the Calabi conjecture which are not yet completely understood, e.g., for non-compact manifolds or singular volume-forms. To approach this problem, one might consider a modification of the set-up described here, with appropriate boundary conditions at infinity or near the singularities. It would be very interesting to know if this moment-map approach sheds any light on these versions of the Calabi conjecture.

Another use of this interpretation of the Calabi conjecture may be as a testing ground for approaches to another famous---and as yet unresolved---conjec\-ture in Kähler geometry, namely the Donaldson--Tian--Yau conjecture concerning the existence of constant scalar curvature Kähler metrics (see \cite{donaldson:scasotv} for a formulation of this conjecture). Since the observation of Donaldson \cite{donaldson-fields} and Fujiki \cite{fujiki-msopamakm} that this problem can be described in terms of a moment-map, the general framework of such problems has guided much work on the subject. 

With this in mind, one may attempt to reprove facts about the Calabi conjecture, directly using the moment-map formalism, and in doing so learn more about the harder problem of constant scalar curvature. Whilst instability does not play a role in the Calabi conjecture (since a solution always exists) the comparison with constant scalar curvature metrics is certainly not devoid of interest. For example, just as the constant scalar curvature problem has a sequence of finite dimensional approximations (involving Bergman spaces and balanced embeddings, see \cite{donaldson-1}) so does the Calabi conjecture (see \cite{donaldson-snricdg}). If one could somehow use the finite dimensional approximations to re-solve problems related to the Calabi conjecture, this may shed light on exactly how to approach constant scalar curvature metrics in an analogous way. 

To be more precise, we give one instance of how this might work. The flow (\ref{flow}) associated to the Calabi conjecture is known to exist for all time, but the present proofs rest on Yau's estimates. These in turn rely on the maximum principle and hence depend critically on the fact that the flow is second order. The analogous flow in the case of constant scalar curvature metrics---the Calabi flow---is \emph{fourth} order and so it is far from clear how to approach it analytically. It is for this reason that long-term existence of the Calabi flow is still an open problem.

It may be possible instead to understand the flow via a sequence of finite dimensional flows on Bergman spaces. In \cite{fine-cflow} a sequence of flows on the Bergman spaces are defined and it is shown that the finite dimensional flows converge to Calabi flow for as long as it exists. Cao and Keller \cite{cao-keller} have very recently proved the analogous result in the case of the flow (\ref{flow}). Now, if one could prove directly that the finite dimensional flows of \cite{cao-keller} converge, one would have a new proof of the long-time existence of the flow (\ref{flow}) which was independent of Yau's estimates and, moreover, written in such a way as to stand a chance of generalising to the case of the more difficult and currently rather intractable Calabi flow. 

{\small
\bibliographystyle{alpha}
\bibliography{conns_symp_curv_refs}
}

{\small \noindent {\tt joel.fine@ulb.ac.be } \newline
D\'epartment de Math\'ematique,
Universit\'e Libre de Bruxelles CP218,\\ 
Boulevard du Triomphe,
Bruxelles 1050,
Belgique.\\}

\end{document}